\documentclass[12pt]{amsart}

\usepackage[colorlinks=true,linkcolor=blue]{hyperref}
\usepackage{amsmath}
\usepackage{amssymb}
\usepackage{amsthm}

\usepackage{graphicx}
\usepackage{epstopdf}
\usepackage{epsfig}

\usepackage{enumitem}

\usepackage{hyperref}

\usepackage{times}
\usepackage{relsize}
\usepackage{textcomp}
\usepackage{amssymb}
\usepackage[english]{babel}
\usepackage[autostyle]{csquotes}
\usepackage{epstopdf}
\usepackage{nicefrac}
\usepackage{tikz}

\theoremstyle{plain} 
\newtheorem{thm}{Theorem}[section]
\newtheorem{prop}[thm]{Proposition}
\newtheorem{lemma}[thm]{Lemma}
\newtheorem{cor}[thm]{Corollary} 
\newtheorem{question}[thm]{Question}

\newtheorem{conj}[thm]{Conjecture}
\theoremstyle{remark}
\newtheorem{remark}[thm]{Remark}

\theoremstyle{definition}

\newcommand{\CC}{\mathbb{C}}

\newcommand{\NN}{\mathbb{N}}

\newcommand{\QQ}{\mathbb{Q}}

\newcommand{\ZZ}{\mathbb{Z}}

\newcommand{\calO}{{\mathcal O}}

\newcommand{\Qbar}{\overline{\QQ}}

\newcommand{\Dbar}{\overline{D}}

\newcommand{\Cv}{\CC_v}
\newcommand{\Cp}{\CC_p}
\newcommand{\Qp}{\QQ_p}

\DeclareMathOperator{\Orb}{Orb}

\newcommand{\frp}{\mathfrak{p} }

\newcommand{\la}{\langle}
\newcommand{\ra}{\rangle}

\newcommand{\dsps}{\displaystyle}

\begin{document}

\title{Misiurewicz polynomials and dynamical units, Part I}
\date{May19, 2022}
\subjclass[2010]{37P15, 11R09, 37P20}
\author[Benedetto]{Robert L. Benedetto}
\address{Amherst College \\ Amherst, MA 01002 \\ USA}
\email{rlbenedetto@amherst.edu}
\author[Goksel]{Vefa Goksel}
\address{University of Massachusetts \\ Amherst, MA 01002 \\ USA}
\email{goksel@math.umass.edu}

\begin{abstract}
We study the dynamics of the unicritical polynomial family $f_{d,c}(z)=z^d+c\in \CC[z]$.
The $c$-values for which $f_{d,c}$ has a strictly preperiodic postcritical orbit
are called \emph{Misiurewicz parameters},
and they are the roots of \emph{Misiurewicz polynomials}.
The arithmetic properties of these special parameters have found applications in both 
arithmetic and complex dynamics.
In this paper, we investigate some new such properties.
In particular, when $d$ is a prime power and $c$ is a Misiurewicz parameter,
we prove certain arithmetic relations between the points in the postcritical orbit of $f_{d,c}$.
We also consider the algebraic integers obtained by evaluating a Misiurewicz polynomial
at a different Misiurewicz parameter, and we ask when these algebraic integers are algebraic units.
This question naturally arises from some results recently proven by Buff, Epstein, and Koch
and by the second author.
We propose a conjectural answer to this question, which we prove in many cases.
\end{abstract}

\maketitle

\section{Introduction}
Let $f\in \mathbb{C}(z)$ be a rational function.
We denote by $f^n$
the iterates of $f$ by composition, i.e., $f^0(z):=z$, and for each $n\geq 1$,
$f^n:=f\circ f^{n-1}$.
Then $f$ and its iterates map $\mathbb{P}^1(\mathbb{C})=\mathbb{C}\cup\{\infty\}$ to itself. The (\emph{forward}) \emph{orbit}
of a point $x\in \mathbb{P}^1(\mathbb{C})$ is
\[\Orb_f^+(x) := \{f^n(x) : n\geq 0\} .\]

We say that $x\in \mathbb{P}^1(\mathbb{C})$ is \emph{periodic} (of period $n$)
if there is an integer $n\geq 1$ such that $f^n(x)=x$;
in that case, the smallest such integer is the \emph{exact} period of $x$.
More generally, we say $x$ is \emph{preperiodic} if there is some $m\geq 0$
such that $f^m(x)$ is periodic. Equivalently, $x$ is preperiodic if and only if
the orbit of $x$ is finite.
In that case, the smallest $m\geq 0$ such that
$f^m(x)$ is periodic is the \emph{tail length} of $x$.
We say $x$ is \emph{preperiodic of type} $(m,n)$ if $x$ is preperiodic with tail length $m$,
and $n$ is the exact period of $f^m(x)$. That is, we have $f^{m+n}(x)=f^m(x)$
for minimal integers $m\geq 0$ and $n\geq 1$. 

The \emph{critical points} of $f$ are the ramification points of $f$ in $P^1(\mathbb{C})$.
We say that $f$ is \emph{postcritically finite}, or \emph{PCF}, if all of the critical
points of $f$ are preperiodic. If $f\in \mathbb{C}[z]$ is a polynomial, then the critical points of $f$
consist of the point at $\infty$ (which is fixed by $f$) and all the roots of $f'$ in $\mathbb{C}$.
Thus, a polynomial is PCF if and only if all the roots of its derivative are preperiodic.

In this paper, we consider the case of unicritical polynomials,
i.e., polynomials with a single finite critical point (of high multiplicity).
After a change of coordinates, we assume this
critical point is $0$. Thus, throughout the paper, we fix an integer $d\geq 2$,
and we define
\[ f(z) := f_c(z) := f_{d,c}(z) := z^d+c. \]
We may consider $f$ as an element of the two-variable polynomial ring $\ZZ[c,z]$,
but we usually consider $c$ to be a parameter,
and we iterate $f$ in the variable $z$ only. That is,
\[ f^2(z) = (z^d+c)^d + c, \quad f^3(z) = \big( (z^d+c)^d + c \big)^d + c , \quad \ldots . \]
For each integer $i\geq 0$, define the polynomial $a_i(c)\in\ZZ[c]$ by
$a_i(c) := f^i(0)$. Thus, the sequence
\[ a_1=c, \quad a_2 = c^d + c, \quad a_3 = (c^d+c)^d + c, \quad \cdots \]
gives the iterates of the critical point $0$ under $f$.
We are interested in the case that $f$ is PCF, i.e., that this orbit is finite.

To this end, fix a  $d$-th root of unity $\zeta$ that is \emph{not} $1$.
For any integers $m\geq 2$ and $n\geq 1$,
we define the \emph{$(m,n)$-Misiurewicz polynomial}
$G_{d,m,n}^{\zeta}(c)\in \ZZ[\zeta][c]$ to be
\begin{equation}
\label{eq:Gdef}
G_{d,m,n}^{\zeta}(c):= \prod_{k|n} \big( a_{m+k-1} - \zeta a_{m-1}\big)^{\mu(n/k)}
\cdot \begin{cases}
\prod_{k|n} \big( a_{k}\big)^{-\mu(n/k)}
& \text{ if } n|m-1,
\\
1 & \text{ if } n \nmid m-1,
\end{cases}
\end{equation}
where $\mu$ denotes the M\"{o}bius $\mu$-function.
A priori, $G^{\zeta}_{d,m,n}$ is a rational function in $\QQ(\zeta)(c)$, but in fact it 
is a monic polynomial in $\ZZ[\zeta][c]$, as we prove in Section~\ref{sec:poly_proof}.
Its roots are parameters $c_0$, called \emph{Misiurewicz parameters},
for which $f_{c_0}^{m+n}(0)=f_{c_0}^m(0)$ but no earlier iterates $f_{c_0}^i(0)$ coincide;
we say $f_{c_0}$ is \emph{PCF of exact type~$(m,n)$}.
The root of unity $\zeta$ further specifies that
$f_{c_0}^{m+n-1}(0)/ f_{c_0}^{m-1}(0)=\zeta$.
Milnor \cite[Remark 3.5]{Milnor12} conjectured that related polynomials over $\QQ$ are irreducible,
and we make the following corresponding conjecture over $\QQ(\zeta)$:
\begin{conj}
\label{conj:misiurewicz_irr}
Let $d,m\geq 2, n\geq 1$, and $\zeta\neq 1$ a $d$-th root of unity. Then $G_{d,m,n}^{\zeta}$ is irreducible over $\mathbb{Q}(\zeta)$
\end{conj}
Recent progress in \cite{BEK19,Gok19,Gok20} has proven Conjecture~\ref{conj:misiurewicz_irr}
in the case that $d=2$ and $n\leq 3$, but otherwise, very little is currently known.
Such arithmetic questions have dynamical consequences, as illustrated
by the work of Buff, Epstein and Koch in \cite{BEK19}, who applied these known instances of
Conjecture~\ref{conj:misiurewicz_irr} to prove the first cases
of a different conjecture of Milnor \cite{Milnor93,Milnor09} on the irreducibility
of certain moduli curves arising in complex dynamics \cite[Theorem 1, Theorem 4]{BEK19}.
See also \cite[Section 2]{HT15} for a brief survey of known results on Misiurewicz parameters,
and our companion paper \cite{BG2} for further results in the study of their arithmetic properties.

More broadly, postcritically finite polynomials play a fundamental role in polynomial dynamics.
On the complex dynamical side, Douady and Hubbard \cite[Chapter 8]{DH}
proved that Misiurewicz parameters are dense in the boundary of the Mandelbrot set,
and Favre and Gauthier \cite[Theorem 1]{FG15} further proved that they are equidistributed in an
appropriate sense. Ghioca, Krieger, Nguyen, and Ye \cite[Theorem 3.1]{GKNY17}
generalized this equidistribution to PCF maps in arbitrary dynamical moduli spaces.
Indeed, as proposed by Baker and DeMarco in \cite{BD13}, PCF maps should play
a role in dynamical moduli spaces analogous to CM points on modular curves,
and more generally to special points on Shimura varieties.

Returning to the unicritical family $f_{d,c}$,
while Misiurewicz parameters are ones for which the critical point is strictly preperiodic,
those for which the critical point is periodic are roots of \emph{Gleason polynomials}.
Specifically, for $n\geq 1$, the roots of the Gleason polynomial
\begin{equation}
\label{eq:Gleasondef}
G_{d,0,n}(c) := \prod_{k|n} (a_k)^{\mu(n/k)} \in \ZZ[c]
\end{equation}
are parameters $c_0$ for which the critical point $0$ has exact period $n$ under $f_{c_0}$.
See, for example, \cite{Buff18,BFKP21,HT15}.
\begin{question}
	\label{question:Gleason_irr}
For which $d\geq 2, n\geq 1$, is the Gleason polynomial $G_{d,0,n}$ irreducible over $\QQ$?
\end{question}
As with Conjecture~\ref{conj:misiurewicz_irr}, very little is known about Question~\ref{question:Gleason_irr}.
Even for fixed degree $d$, there is no infinite family of Gleason polynomials which are
known to be irreducible.
When $d=2$, calculations for small periods suggest that $G_{2,0,n}$ is irreducible over $\QQ$
for all $n\geq 1$, but this conjecture remains wide open.
Buff \cite[Proposition 5]{Buff18} observed that the corresponding conjecture is false in general
by showing that $G_{d,0,3}$ has $2$ irreducible factors if $d\equiv 1\pmod{6}$.
\begin{remark}
The definitions of Misiurewicz and Gleason polynomials are not entirely
consistent in the literature. For example, some authors use the family of maps $az^d+1$,
as in \cite{Buff18,BEK19,BFKP21}, instead of $z^d+c$.
In addition, our choice of a root of unity $\zeta$ is another difference
both from those authors and from previous work in \cite{Gok19,Gok20,HT15}.
\end{remark}

In this paper, we consider various arithmetic properties of the orbits $\Orb_f^{+}(c_0)$,
where $c_0$ is a Misiurewicz parameter. Theorem~\ref{thm:1.1}, which we prove using
purely local methods, concerns the case that the degree $d$ is a prime power.
In particular, it generalizes \cite[Theorem~3.1]{Gok19} from prime degrees to prime-power degrees,
and it answers a question raised in \cite{Gok19}.
Here and throughout the paper, when $K$ is a number field, we denote by $\calO_K$
the ring of integers of $K$, and for any $b\in\calO_K$, we write $\la b\ra$ for
the principal ideal generated by $b$.

\begin{thm}
\label{thm:1.1}
Let $d=p^e$ be a prime power, and suppose $f=f_{d,c_0}$ is PCF of exact type $(m,n)$ with $m > 0$.
Let $\zeta:=a_{m+n-1}(c_0)/a_{m-1}(c_0)\neq 1$, and
let $0\leq r\leq e-1$ be the smallest nonnegative integer such that $\zeta^{p^{r+1}}=1$.
Let $K:=\QQ(c_0)$.
Then:
\begin{enumerate}[label=\textup{(\alph*)}]
	\item If $n\nmid m-1$, then 
	$\dsps \la a_i(c_0) \ra ^{p^r (p-1)d^{m-1}} =\begin{cases}
	\la p \ra & \text{ if } n | i, \\
	\calO_K & \text{ if } n\nmid i .
	\end{cases}$
	\item If $n| m-1$,  then 
	$\dsps \la a_i(c_0) \ra ^{p^r (p-1)(d^{m-1}-1)} =\begin{cases}
	\la p \ra & \text{ if } n | i, \\
	\calO_K & \text{ if } n\nmid i .
	\end{cases}$
\end{enumerate}
\end{thm}
Theorem~\ref{thm:1.1} immediately implies the following new irreducibility result:
\begin{cor}
\label{cor:irr_corollary}
Let $m\geq 2$, and $\zeta\neq 1$ a $d$-th root of unity. Suppose that $d=p^e$ is a prime power. Then $G_{d,m,1}^{\zeta}$ is irreducible over $\mathbb{Q}(\zeta)$.
\end{cor}
\begin{proof}
Let $K=\mathbb{Q}(c_0)$ for a root $c_0$ of $G_{d,m,1}^{\zeta}$. Take a prime ideal $\mathfrak{p}\subseteq \mathcal{O}_K$ which lies over $p$.
By Theorem~\ref{thm:1.1}(b),
the ramification index $e(\mathfrak{p}|p)$ satisfies
$e(\mathfrak{p}|p)\geq p^r(p-1)(d^{m-1}-1)$,
and hence $[\mathbb{Q}(c_0):\mathbb{Q}]\geq p^r(p-1)(d^{m-1}-1)$.
On the other hand, we also have $[\mathbb{Q}(\zeta):\mathbb{Q}] = p^r(p-1)$.
Therefore, $[\mathbb{Q}(c_0):\mathbb{Q}(\zeta)]\geq d^{m-1}-1$.
Since $\text{deg}(G_{d,m,1}^{\zeta}) = d^{m-1}-1$ by direct computation,
this forces $G_{d,m,n}^{\zeta}$ to be irreducible over $\mathbb{Q}(\zeta)$, as desired.
\end{proof}
Let $c_0$ be a root of the Gleason polynomial $G_{d,0,n}$, and set $K=\mathbb{Q}(c_0)$.
The second author showed \cite[Lemma 3.1]{Gok19} that $G_{d,0,i}(c_0)$, i.e., another
Gleason polynomial evaluated at $c_0$, is an algebraic unit in $\calO_K$ unless $i=n$. Buff, Epstein, and Koch studied the resultants of Misiurewicz polynomials with Gleason polynomials,
and they proved that a Misiurewicz polynomial evaluated at a Gleason parameter is an algebraic unit unless the periods of these two polynomials are same \cite[Lemma 26]{BEK19}. They have used these resultants to prove new irreducibility results for Misiurewicz polynomials. In this paper, we study the next natural question for Misiurewicz polynomials:
\begin{question}
\label{ques:unit}
Fix $d,m\geq 2$, $n\geq 1$, and $\zeta\neq 1$ a $d$-th root of unity.
Let $c_0$ be a root of $G_{d,m,n}^\zeta$, and
let $K:=\QQ(c_0)$.
For which integers
$j\geq 2$ and $\ell\geq 1$ is $G_{d,j,\ell}^{\zeta}(c_0)$ an algebraic unit in $\calO_K$?
\end{question}
Note that in the setting of Question~\ref{ques:unit}, we have
$\zeta\in K$, because $\zeta=a_{m+n-1}(c_0)/a_{m-1}(c_0)$.

Question~\ref{ques:unit} is also motivated in part by analogy with the theory of cyclotomic polynomials.
Specifically, the following classical result is well known and has several different proofs in the literature,
the earliest of which is due to Emma T.\ Lehmer \cite[Theorem 4]{Lehmer30}.
\begin{thm}
\label{thm:classical_cyclotomic}
Let $m>n\geq 1$. Denote by $\Phi_m$ the $m$-th cyclotomic polynomial.
Suppose that $\zeta$ is a primitive $n$-th root of unity.
Then  $\Phi_m(\zeta)$ is not an algebraic unit in $\ZZ[\zeta]$
if and only if $m=p^k n$ for some prime $p$ and some integer $k\geq 1$.
\end{thm}

Question~\ref{ques:unit} is also evocative of the study of dynamical units
introduced by Morton and Silverman in \cite{MorSil95}.
However, whereas Morton and Silverman considered units arising from differences between
periodic points of a \emph{single} map $f$, the units and non-units we consider
in this paper arise from \emph{parameters} in a dynamical moduli space.

When $d$ is a prime power, we are able to give the following
answer to Question~\ref{ques:unit} in the case $j\neq m$.

\begin{thm}
\label{thm:jnotm}
Let $d=p^e$, where $p$ is a prime and $e\geq 1$.
Let $c_0$ be a root of $G_{d,m,n}^\zeta$ for some $m\geq 2$, $n\geq 1$,
and $\zeta\neq 1$ a $d$-th root of unity.
Let $K:=\QQ(c_0)$. Suppose that $\ell \geq 1$ and $j\geq 2$.
\begin{enumerate}[label=\textup{(\alph*)}]
\item
If $\ell\neq n$ and $j\neq m$, then $\la G_{d,j,\ell}^{\zeta}(c_0)\ra=\calO_K$.
\item
If $\ell = n$ and $j<m$, then $\la G_{d,j,\ell}^{\zeta}(c_0)\ra = \la a_n\ra^{N_{j,n}}$, where
\[ N_{j,n} = \begin{cases}
d^{j-1} & \text{ if } n  \nmid j-1, \\
d^{j-1}-1 & \text{ if } n  \,\mid\, j-1.
\end{cases} \]
\item
If $\ell = n$ and $j>m$, then $\la G_{d,j,\ell}^{\zeta}(c_0)\ra=\la 1-\zeta\ra$.
\end{enumerate}
\end{thm}

When $j=m$, Magma computations suggest the following conjecture.

\begin{conj}
\label{conj:j=m}
Let $d=p^e$, where $p$ is a prime and $e\geq 1$. Let $c_0$ be a root of $G_{d,m,n}^\zeta$
for some $m\geq 2$, $n\geq 1$, and $\zeta\neq 1$ a $d$-th root of unity.
Set $K:=\mathbb{Q}(c_0)$. Suppose that $1\leq\ell\leq n$. Then
\[ G_{d,m,\ell}^{\zeta}(c_0) \text{ is a unit in } \mathcal{O}_K
\quad\text{ if and only if }\quad
\ell\nmid n.\]
\end{conj}
The techniques needed to analyze the case $j=m$ are very different from those
used in the current paper for $j\neq m$.
Therefore we discuss the above conjecture in greater detail in
the sequel paper \cite{BG2}.

The structure of the paper is as follows.
In Section~\ref{sec:poly_proof}, we prove that $G_{d,m,n}^{\zeta}$ is a polynomial.
We prove Theorem~\ref{thm:1.1} in Section~\ref{sec:local},
answering the question posed in \cite{Gok19} in the affirmative.
We then consider the $j\neq m$ case of Question~\ref{ques:unit},
proving Theorem~\ref{thm:jnotm} for $j<m$ in Section~\ref{sec:j<m},
and for $j>m$ in Section~\ref{sec:j>m}.

\section{$G_{d,m,n}^{\zeta}$ is a polynomial}
\label{sec:poly_proof}
The purpose of this section is to prove the following theorem.
\begin{thm}
\label{thm:G_{d,m,n}_poly}
Let $d,m\geq 2$, $n\geq 1$ be integers,
and let $\zeta\neq 1$ be a $d$-th root of unity.
Then $G_{d,m,n}^{\zeta}$ is a monic polynomial in $\ZZ[\zeta][c]$,
with only simple roots.
\end{thm}

Our proof will require two auxiliary lemmas, as follows.

\begin{lemma}
\label{lem:vanish_order}
Let $m,n,d,\zeta$ be as in Theorem \ref{thm:G_{d,m,n}_poly}.
Let $\alpha\in\Qbar$ satisfy $a_{m+k-1}(\alpha)=\zeta a_{m-1}(\alpha)$
for some positive divisor $k$ of $n$, and suppose that $k$ is the smallest
positive divisor of $n$ for which this equality holds.
Then for any integer $\ell |n$, we have
\[ a_{m+\ell-1}(\alpha) = \zeta a_{m-1}(\alpha) \text{ }\iff\text{ } k|\ell.\]
\end{lemma}

\begin{proof}
Write $f:=f_{d,\alpha}$.
By definition of $k$, we have $a_{m+k-1}(\alpha) = \zeta a_{m-1}(\alpha)$.
Applying $f^k$ to both sides of this equality, we obtain $a_{m+2k-1}(\alpha) = a_{m+k-1}(\alpha)$.
Applying $f^k$ repeatedly, then, we have $a_{m+ik-1}(\alpha) = a_{m+k-1}(\alpha)$
for any integer $i\geq 1$.

Armed with this fact, we can now prove the equivalence.
For the reverse implication, i.e., assuming $k|\ell$, we have $\ell=ik$ for some $i\geq 1$, and hence
\[a_{m+\ell-1}(\alpha) = a_{m+k-1}(\alpha) = \zeta a_{m-1}(\alpha),\]
as desired.

For the forward implication, we assume $a_{m+\ell-1}(\alpha)=\zeta a_{m-1}(\alpha)$.
There exist positive integers $i,j,t\geq 1$ such that $ik + j\ell = tk +\gcd(k,\ell)$.
As we saw at the start of this proof, we have $a_{m+ik-1}(\alpha)=\zeta a_{m-1}(\alpha)$;
applying $f^{j\ell}$ yields
\[ a_{m+ik+j\ell -1}(\alpha) = a_{m+j\ell -1}(\alpha) =\zeta a_{m-1}(\alpha). \]
On the other hand, by our choice of $i,j,t$, we have
\[ a_{m+ik+j\ell -1}(\alpha) = a_{m+tk+\gcd(k,\ell)-1}(\alpha) = a_{m+\gcd(k,\ell)-1}(\alpha),\]
so that $a_{m+\gcd(k,\ell)-1}(\alpha) = \zeta a_{m-1}(\alpha)$.
But $k$ was the smallest positive divisor of $n$ satisfying $a_{m+k-1}(\alpha)=\zeta a_{m-1}(\alpha)$,
and since $1\leq \gcd(k,\ell)\leq k$ is also a divisor of $n$, we must have $\gcd(k,\ell)=k$.
That is, $k|\ell$, as desired.
\end{proof}

\begin{lemma}
\label{lem:common_root}
Let $m,n,d,\zeta$ be as in Theorem \ref{thm:G_{d,m,n}_poly},
and suppose that $n|m-1$.
Let $\alpha$ be any root of $G_{d,0,n}$, and let $i|n$ be a positive integer divisor of $n$.
Then $a_{m+i-1}(\alpha) = \zeta a_{m-1}(\alpha)$ if and only if $i=n$.
\end{lemma}

\begin{proof}
Applying M\"{o}bius inversion to the definition of Gleason polynomials from
\eqref{eq:Gleasondef}, for any positive integer $t\geq 1$, we have
\begin{equation}
\label{eq:Mobius_Gleason}
a_{nt} = \prod_{k|nt} G_{d,0,k} = a_n \prod_{\substack{ k|nt \\ k \nmid n}} G_{d,0,k}.
\end{equation}
In particular, in the polynomial ring $\ZZ[c]$, we have $a_n | a_{nt}$ and $G_{d,0,n} | a_n$.
Thus, we have $a_{nt}(\alpha)=0$, since $G_{d,0,n}|a_{nt}$ and $\alpha$ is a root of $G_{d,0,n}$.

To prove the desired equivalence, we begin with the reverse implication, i.e., we suppose that $i=n$.
Because we have $n|m-1$ and hence also $n|m+i-1$, it follows that
$a_{m+i-1}(\alpha) = 0= \zeta a_{m-1}(\alpha)$, as desired.

Conversely, suppose $a_{m+i-1}(\alpha)=\zeta a_{m-1}(\alpha)$.
Because $n|m-1$, we have $a_{m-1}(\alpha)=0$, and hence $a_{m+i-1}(\alpha)=0$ as well.
Therefore, we have
\[ 0= f^{m+i-1}(0) = f^i\big( f^{m-1}(0) \big) = f^i(0), \]
or equivalently, $a_i(\alpha)=0$.
However, $\alpha$ was a root of $G_{d,0,n}$,
and $G_{d,0,n}$ is known to be relatively prime to $a_i$ for $1\leq i<n$.
(See, for instance, \cite[Lemma 30]{BEK19}, which shows that
the resultant of two different Gleason polynomials is $\pm 1$, and hence they share no roots.
Since $a_i$ is a product of Gleason polynomials, it is indeed relatively prime to $G_{d,0,n}$.)
Thus, we must have $i=n$, as desired.
\end{proof}

\begin{proof}[Proof of Theorem \ref{thm:G_{d,m,n}_poly}]
\textbf{Case 1}.
Suppose that $n\nmid m-1$. By definition, we have
\begin{equation}
\label{eq:Gdef1}
G_{d,m,n}^{\zeta} = \prod_{i|n} (a_{m+i-1}-\zeta a_{m-1})^{\mu(n/i)}.
\end{equation}
Let $\alpha$ be a root of $a_{m+k-1}-\zeta a_{m-1}$ for some minimal positive integer $k|n$.
By Lemma \ref{lem:vanish_order}, for any positive divisor $\ell$ of $n$, we have that
$\alpha$ is a root of $a_{m+\ell-1}-\zeta a_{m-1}$ if and only if $k|\ell$.
In that case, as shown in the proof of Theorem~A.1 of \cite{Epstein12},
the order of vanishing of $a_{m+\ell-1}-\zeta a_{m-1}$ at $\alpha$ is $1$.
Thus, the order of vanishing of $G_{d,m,n}^{\zeta}$ at $\alpha$ is
\[ \sum_{\substack{\ell | n \\ k | \ell}} \mu\bigg(\frac{n}{\ell} \bigg)
=\sum_{t|(n/k)} \mu\bigg(\frac{n/k}{t} \bigg) = \begin{cases}
1 & \text{ if } k=n, \\
0 & \text{ if } k<n,
\end{cases} \]
where we have applied the well-known identity
\begin{equation}
\label{eq:Mobius}
\sum_{t|N} \mu\bigg(\frac{N}{t}\bigg) = \begin{cases}
1 & \text{if } N=1\\
0 & \text{if } N>1.
\end{cases} 
\end{equation}
Thus, the rational function $G_{d,m,n}^{\zeta}$ has order of vanishing either $0$ or $1$
at every point of $\Qbar$.
It follows that $G_{d,m,n}^{\zeta}$ is a polynomial in $\QQ(\zeta)[c]$,
and it has only simple roots.
Finally, because all of the multiplicands in equation~\eqref{eq:Gdef1}
are monic polynomials in $\ZZ[\zeta][c]$, the polynomial
$G_{d,m,n}^{\zeta}$ is also a monic and lies in $\ZZ[\zeta][c]$.

\textbf{Case 2}.
Suppose that $n|m-1$. By definition, we have
\[G_{d,m,n}^{\zeta} =
\frac{\prod_{i|n} (a_{m+i-1}-\zeta a_{m-1})^{\mu(n/i)}}{\prod_{i|n} a_i^{\mu(n/i)}}.\]
As we saw in Case~1, the numerator is a monic polynomial in $\ZZ[\zeta][c]$ with simple roots,
so we only need to consider roots of $G_{d,0,n}= \prod_{i|n} a_i^{\mu(n/i)}$,
which is also known to have simple roots
(see, for instance, \cite[Lemma~19.1]{DH} or \cite[Proposition A.1]{Epstein12}). 

For any root $\alpha$ of $G_{d,0,n}$, Lemma~\ref{lem:common_root} says that the only
term of the numerator that has $\alpha$ as a root is when $i=n$,
i.e., the term
$(a_{m+n-1}-\zeta a_{m-1})^{\mu(n/n)} = a_{m+n-1}-\zeta a_{m-1}$,
which has a (simple) root at $\alpha$.
Thus, $G_{d,m,n}^{\zeta}$ has order of vanishing zero at $\alpha$.
As before, then, it follows that $G_{d,m,n}^{\zeta}$ is a monic polynomial in $\ZZ[\zeta][c]$,
with only simple roots.
\end{proof}

\section{Local results}
\label{sec:local}
The results of this section generalize estimates proven by the second author in \cite{Gok19}.
Throughout this section, fix integers $d,m,n$ with $d,m\geq 2$ and $n\geq 1$.
Let $c_0\in\Qbar$ be a Misiurewicz parameter of exact type $(m,n)$,
write $f:=f_{d,c_0}$, and define $K:=\QQ(c_0)$.

For any finite place $v$ of $K$, we define $K_v$ to be the $v$-adic completion of $K$,
and $\Cv$ to be the completion of an algebraic closure of $K_v$. For any $x\in\Cv$ and $r>0$,
we denote by
\[ D(x,r) = \{y\in\Cv \, : \, |y-x|_v < r \} \]
the open disk of radius $r$ centered at $x$ in $\Cv$.

We begin with the following modest strengthening of \cite[Lemma~2.4]{Gok19}.

\begin{prop}
\label{prop:vefa2.4}
If $f$ is PCF of exact type $(m,n)$, then for every finite place $v$ of $K$, either:
\begin{itemize}
\item $v(a_i(c_0))=0$ for all $i\geq 1$, or
\item $v(a_n(c_0))>0$, and for all $i\geq 1$, we have
$\dsps v(a_i(c_0)) = \begin{cases}
v(a_n(c_0)) & \text{ if } n|i, \\
0 & \text{ if } n\nmid i.
\end{cases}$
\end{itemize}
\end{prop}

Applying Proposition~\ref{prop:vefa2.4} at every finite place $v$ of $K$, we immediately obtain:

\begin{cor}
\label{cor:2.4}
For every $i\geq 1$, $\dsps \la a_i(c_0)\ra = \begin{cases}
\la a_n(c_0)\ra & \text{ if } n|i, \\
\calO_K & \text{ if } n\nmid i.
\end{cases}$
\end{cor}

\begin{proof}[Proof of Proposition~\ref{prop:vefa2.4}]
We already know $v(a_i(c_0))\geq 0$ for all $i\geq 1$.
If $v(a_i(c_0))=0$ for all $i$, then we are in the first case, and we are done.
So we assume for the remainder of the proof that $v(a_{\ell}(c_0))>0$ for some minimal $\ell \geq 1$.

Thus, $f^\ell$ maps $D(0,1)$ onto itself multiply-to-1, and hence by Theorem~4.18(b) of \cite{BenBook},
the disk $D(0,1)$ contains a unique periodic point $b$ of $f$, which is $v$-adically attracting and of exact period $\ell$.
Because $f^m(0)$ is a periodic point of exact period $n$ lying in $f^m(D(0,1))$,
it must be in the same cycle as $b$, and hence $\ell=n$. 

Since the disk $D(0,1)$ has exact period $\ell=n$, we have $v(a_i(c_0))=0$ for all $i\geq 1$ for which $n\nmid i$.
It remains to consider $i$ of the form $i=nj$ for $j\geq 1$.

If $b=0$, then $z=0$ itself is periodic, so $m=0$, and we have $a_n(c_0)=a_i(c_0)=0$, and we are done.
Thus, we assume for the rest of the proof that $b\neq 0$.

Because the periodic point $b\neq 0$ is attracting, we have
\[ |0-b|_v > |a_n(c_0)-b|_v \geq  |a_{2n}(c_0) -b|_v \geq |a_{3n}(c_0)-b|_v \geq \cdots, \]
and hence $|a_{nj}(c_0)|_v = |b|_v$ for all $j\geq 1$. In particular, writing $i=nj$, we have
\[ v(a_i(c_0)) = v(a_{nj}(c_0)) = v( a_n(c_0)) \]
as desired.
\end{proof}

We have $a_{m+n-1}(c_0)^d=a_{m-1}(c_0)^d$ but $a_{m+n-1}(c_0)\neq a_{m-1}(c_0)$,
and hence there is a $d$-th root of unity $\zeta\neq 1$ such that $a_{m+n-1}(c_0)=\zeta a_{m-1}(c_0)$.
We also have $a_{m-1}(c_0)\neq 0$.

\begin{thm}
\label{thm:vefa3.1}
Fix a finite place $v$ of $K$. Suppose $f=f_{d,c_0}$ is PCF of exact type $(m,n)$,
with $m> 0$. Then:
\begin{enumerate}[label=\textup{(\arabic*)}]
\item If $v(d)=0$, then $v(a_i(c_0))=0$ for all $i\geq 1$.
\item If $d=p^e$ is a prime power,
let $\zeta:=a_{m+n-1}(c_0)/a_{m-1}(c_0)\neq 1$.
Let $0\leq r\leq e-1$ be the smallest nonnegative integer such that $\zeta^{p^{r+1}}=1$.
\begin{enumerate}[label=\textup{(\alph*)}]
\item If $n\nmid m-1$, then
$\dsps p^r(p-1) d^{m-1} v(a_i(c_0)) =\begin{cases}
v(p) & \text{ if } n | i, \\
0 & \text{ if } n\nmid i .
\end{cases}$
\item If $n | m-1$, then
$\dsps p^r(p-1) (d^{m-1}-1) v(a_i(c_0)) =\begin{cases}
v(p) & \text{ if } n | i, \\
0 & \text{ if } n\nmid i .
\end{cases}$
\end{enumerate}
\end{enumerate}
\end{thm}

Applying Theorem~\ref{thm:vefa3.1} at every finite place $v$ immediately yields Theorem~\ref{thm:1.1}.


To prove part (2) of Theorem~\ref{thm:vefa3.1}, we will need the following two lemmas.
We denote by $\Cp$ the completion of an algebraic closure of the $p$-adic field $\Qp$.

\begin{lemma}
\label{lem:ppower}
Let $p$ be a prime, let $e\geq 1$ be an integer, and let $d=p^e$.
Let $a,b,c_0\in\Cp$ with $|a|_p=|b|_p>0$,
and define $f(z):=z^d+c_0$. If $|f(a)-f(b)|_p \geq |p|_p^{p/(p-1)}|b|_p^d$, then
$|f(a)-f(b)|_p = |a-b|_p^d$.
\end{lemma}

\begin{proof}
Let $v$ be the valuation on $\Cp$, normalized so that $v(p)=1$.
Let $w:=(a-b)/b$ and $x:= (f(a)-f(b))/b^d$, and define
\[ g(t) := (1+t)^d - 1 = \sum_{i=1}^d \binom{d}{i} t^i \in \ZZ[t]
\quad\text{and}\quad
h(t) := g(t)-x \in \Cp[t] .\]
Then 
\[ h(w) = (1+w)^d - 1 - x =\frac{a^d-b^d}{b^d} -x = 0, \]
i.e., $w$ is a root of the polynomial $h$.
However, the Newton polygon of $g$ has vertices at $(p^r,e-r)$ for $r=0,1,\ldots,e$,
and the hypotheses say that $v(x)\leq p/(p-1)$. Thus, the Newton polygon of $h$ consists
of a single segment of length $d$ and slope $-v(x)/d$.
Hence, the root $w$ satisfies $dv(w)=v(x)$, and therefore $|w|_p^d = |x|_p$.
Multiplying both sides of this equation by $|b|_p^d$ yields the desired result.
\end{proof}

\begin{lemma}
\label{lem:nextpower}
Let $p$ be a prime, let $e\geq 1$ be an integer, let $d=p^e$, let $c_0\in\Cp$,
and suppose that $f(z):=z^d+c_0$ is PCF of exact type $(m,n)$.
Then for every $0\leq i\leq m-2$,
\[ \big|a_{i+n+1}(c_0) - a_{i+1}(c_0) \big|_p = \big|a_{i+n}(c_0) - a_i(c_0)\big|_p^d .\]
\end{lemma}

\begin{proof}
\textbf{Step 1}. We claim that for every $0\leq i\leq m-1$, we have
\begin{equation}
\label{eq:diffbound}
|p|_p^{1/(p-1)} |a_n(c_0)|_p \leq |a_{i+n}(c_0)-a_i(c_0)|_p
\end{equation}
Indeed, if inequality~\eqref{eq:diffbound} fails for any $0\leq i\leq m-1$, then
because $|a_j(c_0)|_p\leq 1$ for all $j$, we have
\[ |a_{i+n+1}(c_0) - a_{i+1}(c_0)|_p = |a_{i+n}(c_0)^d -a_i(c_0)^d|_p
\leq |a_{i+n}(c_0)-a_i(c_0)|_p < |p|_p^{1/(p-1)} |a_n(c_0)|_p, \]
so that the inequality also fails for $i+1$.
By induction, then, it fails for $m-1$, meaning that
\[ |a_{m+n-1}(c_0) - a_{m-1}(c_0)|_p < |p|_p^{1/(p-1)} |a_n(c_0)|_p \le |p|_p^{1/(p-1)} |a_{m-1}(c_0)|_p, \]
where we have used the fact that $|a_{m-1}(c_0)|_p\geq |a_n(c_0)|_p$
by Proposition~\ref{prop:vefa2.4}.
However, both
the map $z\mapsto z^d$, and hence also $f$, are one-to-one on the open disk
\[ D\big(a_{m-1}(c_0),|p|_p^{1/(p-1)}|a_{m-1}(c_0)|_p\big) . \]
But the distinct points $a_{m+n-1}(c_0)$ and $a_{m-1}(c_0)$ both lie in this disk, and they both map to
$a_{m+n}(c_0)=a_m(c_0)$ under $f$.
This contradiction proves our claim.

\medskip

\textbf{Step 2}.
Note that
\begin{equation}
\label{eq:izerod}
|a_{n+1}(c_0) - a_1(c_0)|_p = |f (a_n(c_0)) - f(0)|_p = |a_n(c_0)^d|_p = |a_n(c_0)|_p^d ,
\end{equation}
yielding the desired equality for $i=0$, because $a_0=0$.
Moreover, combining equation~\eqref{eq:izerod} with inequality~\eqref{eq:diffbound},
we have $|p|_p^{1/(p-1)} |a_n(c_0)|_p \leq |a_n(c_0)|_p^d$, and hence
\[ |a_n(c)|_p \geq |p|_p^{1/((p-1)(d-1))} \geq |p|_p .\]
Inequality~\eqref{eq:diffbound} therefore implies
\[ |a_{i+n}(c_0)-a_i(c_0)|_p \geq |p|_p^{p/(p-1)} \quad\text{for all } 0\leq i\leq m-1. \]
In particular, since $|a_j(c_0)|_p\leq 1$ for all $j$, we have
\[ |f(a_{i+n}(c_0))-f(a_i(c_0))|_p \geq |p|_p^{p/(p-1)} |a_i(c_0)|_p^d   \quad\text{for all } 1\leq i\leq m-2. \]
Therefore, we may apply Lemma~\ref{lem:ppower} inductively,
yielding the desired conclusion.
\end{proof}

\begin{proof}[Proof of Theorem~\ref{thm:vefa3.1}]
\textbf{Case (1)}.
Suppose first that $v(d)=0$. If $v( a_n(c_0)) >0$, then again by Theorem~4.18(b) of \cite{BenBook},
there is a unique periodic point $b$ of $f$ in $D(0,1)$, which is $v$-adically attracting and of exact period $n$.
(And we must have $b=f^{nk}(0)$ for some $k\geq 0$ with $nk\geq m$.)
But because $v(d)=0$, we have that $f(z)=z^d+c_0$ is one-to-one on each disk $D(x,|x|)$ for $x\in\Cv^{\times}$.
In particular, $f$ is one-to-one on each disk $D(a_i(c_0),1)$ for $i=1,\ldots, n-1$,
and on the disk $D(b,|b|)$.
Thus, $f^n$ maps $D(0,1)$ $d$-to-$1$ onto $D(0,1)$,
with $D(b,|b|)$ mapping bijectively onto a (proper) subdisk of itself.

If $f^n(0)=b$, then since $f^n(b)=b$ but $b\neq 0$ (because $m\neq 0$),
the inverse image of $b$ under $f^n$ includes $0$ counted with multiplicity $d$,
and $b$ with multiplicity $1$, for a total of (at least) $d+1$, contradicting the fact that $f^n$
has degree $d$ on $D(0,1)$.

On the other hand, if $f^n(0)\neq b$, then because $b$ is attracting, we have $|f^n(0)-b|_v < |0-b|_v$,
so that $f^n(0)\in D(b,|b|)$. But then, because $f^n: D(b,|b|)\to D(b,|b|)$ is one-to-one with $b$ fixed,
the iterates $f^{nj}(0)$ are never equal to $b$ for $j\geq 1$, contradicting the fact that $b=f^{nk}(0)$ for some $k\geq 0$.
Thus, either way, we have a contradiction, and hence our original assumption that
$v(a_n(c_0))>0$ is impossible.
That is, $v(a_n(c_0))=0$. By Proposition~\ref{prop:vefa2.4},
we have $v(a_i(c_0))=0$ for all $i\geq 1$, proving statement (1).

\medskip

\textbf{Case (2)}.
For the remainder of the proof, we may assume that $d=p^e$ is a prime power,
and that $v|p$ (i.e., $v(d)>0$).
The map $f(z)=z^d+c_0$ is a bijection on the residue field, since it is a composition of Frobenius and a translation.
Thus, $f$ acts as a bijection on the (finite) set of open unit disks $\{D(x,1) : x\in\calO_K\}$.
Every disk is therefore periodic (as opposed to just preperiodic) under this action.
In particular, there is some $\ell>0$ such that $f^{\ell}(0)\in D(0,1)$.
By Proposition~\ref{prop:vefa2.4}, then, we have $v(a_n(c_0))>0$,
and $v(a_i(c_0))=v(a_n(c_0))$ if and only if $n|i$.
(And if $n\nmid i$, then $v(a_i(c_0))=0$.)
Thus, it suffices to show the desired formula in the case that $i=n$.

Let $\zeta:=a_{m+n-1}(c_0)/a_{m-1}(c_0)$, and let $0\leq r\leq e-1$ be the smallest nonnegative integer
such that $\zeta^{p^{r+1}}=1$, as in the statement of the theorem.
Then 
\begin{equation}
\label{eq:anform}
|a_n(c_0)|_p^{d^{m-1}} = |a_n(c_0) - 0|_p^{d^{m-1}}
= |a_{m+n-1}(c_0) - a_{m-1}(c_0)|_p = |\zeta-1|_p |a_{m-1}(c_0)|_p
\end{equation}
where the second equality is by repeated application of Lemma~\ref{lem:nextpower}.

If $n\nmid (m-1)$, then $|a_{m-1}(c_0)|_p=1$ by Proposition~\ref{prop:vefa2.4}, 
whence
\[ d^{m-1} v(a_n(c_0)) = v(\zeta-1) = \frac{1}{p^r(p-1)} v(p) , \]
where we have used the well known fact that
\begin{equation}
\label{eq:zeta1}
\big|\zeta-1\big|_p^{p^r(p-1)} = |p|_p .
\end{equation}
Thus, we have the desired equality
\[ p^r (p-1) d^{m-1} v(a_n(c_0)) = v(p) .\]
On the other hand, if $n | (m-1)$,
then $|a_{m-1}(c_0)|_p=|a_n(c_0)|_p$ by Proposition~\ref{prop:vefa2.4},
and therefore equation~\eqref{eq:anform} becomes
\[ d^{m-1} v(a_n(c_0)) = v(\zeta-1) + v(a_n(c_0)) = \frac{1}{p^r (p-1)} v(p) + v(a_n(c_0)), \]
and hence
\[ p^r (p-1)(d^{m-1}-1) v(a_n(c_0)) = v(p), \]
as desired.
\end{proof}

\section{$G_{d,j,\ell}^{\zeta }(c_0)$ when $j<m$.}
\label{sec:j<m}
In this section, we answer Question~\ref{ques:unit} for $j<m$ by proving
Theorem~\ref{thm:jnotm} in that case.
We begin with the following lemma,
which is an analogue of part~(1) of Theorem~\ref{thm:vefa3.1} for the
principal ideal $\la a_{j+\ell-1}(c_0)-\zeta a_{j-1}(c_0)\ra$ when $2\leq j\leq m-1$.

\begin{lemma}
\label{lem:div-d}
Let $d,m\geq 2$ and $n\geq 1$.
Let $\zeta,w\neq 1$ be nontrivial $d$-th roots of unity,
and let $c_0$ be a root of $G_{d,m,n}^{\zeta}$.
Set $L=\QQ(c_0,w)$. Suppose that $2\leq j\leq m-1$ and $\ell\geq 1$.
Then for any prime ideal $\frp$ of $\calO_L$, we have
\[ \frp | \la a_{j+\ell-1}(c_0)-w a_{j-1}(c_0)\ra\implies \frp | \la d\ra. \]
\end{lemma}
\begin{proof}
Applying $f:=f_{d,c_0}$ to both sides of $a_{j+\ell-1}(c_0) \equiv w a_{j-1}(c_0) \pmod{\frp}$ yields
	\begin{equation}
	\label{eq:4}
	a_{j+\ell}(c_0) \equiv a_j(c_0) \pmod{\frp}.
	\end{equation}
Repeatedly applying $f$ to both sides of \eqref{eq:4}, we obtain
	\begin{equation}
	\label{eq:5}
	a_{k+t\ell}(c_0) \equiv a_k(c_0) \pmod{\frp}
	\end{equation}
for any $k\geq j$ and $t\geq 1$. 

In particular, using $k=m-1\geq j$ and $t=n$ in \eqref{eq:5}, we have
	\begin{equation}
	\label{eq:6}
	a_{m-1+n\ell}(c_0) \equiv a_{m-1}(c_0) \pmod{\frp}.
	\end{equation}
Since $c_0$ is a root of $G_{d,m,n}^{\zeta}$, we have $a_{m-1+n\ell}(c_0) = \zeta a_{m-1}(c_0)$.
Substituting this in~\eqref{eq:6}, it follows that
$(\zeta-1)a_{m-1}(c_0)\equiv 0 \pmod{\frp}$, and hence
\begin{equation}
\label{eq:6a}
\text{either} \quad \frp | \la \zeta - 1 \ra
\quad \text{or}\quad \frp | \la a_{m-1}(c_0)\ra .
\end{equation}
It is well known that $\la\zeta-1\ra |\la d\ra$.
Moreover, by \cite[Theorem~1.4]{Gok19}, we have $\la a_{m-1}(c_0)\ra | \la d\ra$.
(Alternatively, if $d$ is a prime power, these two facts are immediate
from~\eqref{eq:zeta1} and our Theorem~\ref{thm:1.1}, respectively.)
The desired result follows immediately from these two facts and \eqref{eq:6a}.
\end{proof}

We also need the following analogue of part~(2) of Theorem~\ref{thm:vefa3.1}
for the same setting as in Lemma~\ref{lem:div-d}, provided $d$ is a prime power.

\begin{prop}
\label{prop:3.2}
Let $m\geq 2$ and $n\geq 1$, and let $d=p^e$, where $p$ is a prime and $e\geq 1$.
Let $\zeta,w\neq 1$ be nontrivial $d$-th roots of unity,
and let $c_0$ be a root of $G_{d,m,n}^{\zeta}$.
Set $L=\QQ(c_0,w)$. Suppose that $2\leq j\leq m-1$ and $\ell\geq 1$.
\begin{enumerate}[label=\textup{(\alph*)}]
\item
If $\ell\not\equiv 0 \pmod{n}$, then $\la a_{j+\ell-1}(c_0)- w a_{j-1}(c_0)\ra=\calO_L$.
\item
If $\ell \equiv 0 \pmod{n}$, then $\la a_{j+\ell-1}(c_0)-\zeta a_{j-1}(c_0)\ra=\la a_n(c_0)\ra^{d^{j-1}}$.
\end{enumerate}
\end{prop}

\begin{proof}
\textbf{Case (a)}.
Let $f=f_{d,c_0}$, and write
	\begin{equation}
	\label{eq:7}
	a_{j+\ell-1}(c_0)- w a_{j-1}(c_0) = f^{j-1}(a_{\ell}(c_0))-w a_{j-1}(c_0).
	\end{equation}
Expanding the expression $f^{j-1}(a_{\ell}(c_0))$, there exists a polynomial $F\in \ZZ[x]$ such that
	\[ f^{j-1}(a_{\ell}(c_0))=a_{\ell}(c_0)^{d^{j-1}}+pF(c_0)+a_{j-1}(c_0). \]
Thus, equation~\eqref{eq:7} becomes
	\begin{equation}
	\label{eq:8}
	a_{j+\ell-1}(c_0)-w a_{j-1}(c_0) = a_{\ell}(c_0)^{d^{j-1}}+pF(c_0)+(1-w)a_{j-1}(c_0).
	\end{equation}
Suppose there were a prime ideal $\frp\subseteq \calO_L$ dividing
$\la a_{j+\ell-1}(c_0)-w a_{j-1}(c_0)\ra$.
Then
\[ a_{\ell}(c_0)^{d^{j-1}}+pF(c_0)+(1-w)a_{j-1}(c_0) \equiv 0 \pmod{\frp} .\]
We have $p\equiv 0 \pmod{\frp}$ by Lemma~\ref{lem:div-d}, and hence
\[ a_{\ell}(c_0)^{d^{j-1}}+(1-w)a_{j-1}(c_0)\equiv 0 \pmod{\frp}. \]

By \eqref{eq:zeta1}, we have $\la 1-w\ra^{p^r(p-1)} = \la p\ra$ as ideals in $\calO_L$,
where $0\leq r\leq e-1$ is the smallest integer such that  $w^{p^{r+1}}=1$.
Hence $1-w \equiv 0 \pmod{\frp}$, which forces $a_{\ell}(c_0)\equiv 0 \pmod{\frp}$.
This contradicts Corollary~\ref{cor:2.4}, which says that $a_{\ell}(c_0)$ is a unit in $\calO_L$.

\textbf{Case (b)}.
Putting $\zeta$ in the role of $w$ in the proof of part~(a), we have
$\la 1-\zeta\ra^{p^r(p-1)} = \la p\ra$,
where $r$ is the same integer as in Theorem~\ref{thm:vefa3.1}. Let
\begin{equation}
\label{eq:EMdef}
E:=p^r(p-1)
\quad\text{and}\quad
M:= \begin{cases}
d^{m-1} & \text{ if } n \nmid (m-1), \\
d^{m-1} - 1 & \text{ if } n | (m-1) ,
\end{cases}
\end{equation}
so that part~(2) of Theorem~\ref{thm:vefa3.1} says $\la a_{\ell}(c_0)\ra^M = \la 1-\zeta \ra$
and $\la a_{\ell}(c_0)\ra^{EM} = \la p \ra$.
Thus, equation~\eqref{eq:8} becomes
\begin{equation}
	\label{eq:9}
	a_{j+\ell-1}(c_0)-\zeta a_{j-1}(c_0)
	= a_{\ell}(c_0)^{d^{j-1}} Q
\end{equation}
where
\begin{equation}
\label{eq:Qdef}
Q := 1+u_1 a_{\ell}(c_0)^{EM-d^{j-1}}F(c_0)+u_2 a_{\ell}(c_0)^{M-d^{j-1}}a_{j-1}(c_0)
\end{equation}
for some units $u_1,u_2$ in $\calO_L$.
Clearly $EM\geq M > d^{j-1}$, so all of the exponents in~\eqref{eq:Qdef} are positive,
and hence $Q\in\calO_L$.

Suppose there were a prime ideal $\frp\subseteq \calO_L$ such that
$Q \equiv 0 \pmod{\frp}$.
Then by~\eqref{eq:9} we would also have $a_{j+\ell-1}(c_0)-\zeta a_{j-1}(c_0)\equiv 0 \pmod{\frp}$,
so that Lemma~\ref{lem:div-d} yields $\frp | \la d \ra$, and hence $\frp | \la a_{\ell}(c_0) \ra$,
since $\la a_{\ell}(c_0)\ra^{eEM} = \la d \ra$.
Equation~\eqref{eq:Qdef} therefore yields $0 \equiv Q \equiv 1 \pmod{\frp}$, a contradiction,
so no such $\frp$ exists.
That is,
$Q$ is a unit in $\calO_L$. Equation~\eqref{eq:9} then implies
$\la a_{j+\ell-1}(c_0)-\zeta a_{j-1}(c_0)\ra = \la a_{\ell}(c_0)\ra^{d^{j-1}} = \la a_n(c_0)\ra^{d^{j-1}}$, as desired. Note that we used Theorem~\ref{thm:1.1} in the last equality.
\end{proof}

We need one more lemma before we can prove Theorem~\ref{thm:jnotm} for $j<m$.

\begin{lemma}
\label{lem:primitive-w}
Let $m\geq 2$ and $n\geq 1$, and let $d=p^e$, where $p$ is a prime and $e\geq 1$.
Let $c_0$ be a root of $G_{d,m,n}^{\zeta}$, where $\zeta\neq 1$ is a $d$-th root of unity.
Let $w$ be a primitive $d$-th root of unity, and set $L=\QQ(c_0,w)$.
Suppose that $2\leq j\leq m-1$ and $\ell\geq 1$.
Then, for any $d$-th root of unity $w'\neq 1$, we have
\[ \la a_{j+\ell-1}(c_0)-w a_{j-1}(c_0)\ra = \la a_{j+\ell-1}(c_0)-w' a_{j-1}(c_0)\ra\]
as ideals in $\calO_L$.
\end{lemma}

\begin{proof}
If $\ell\not\equiv 0(\text{mod }n)$, then by Proposition~\ref{prop:3.2}.(a), we have
\[ \la a_{j+\ell-1}(c_0)-w a_{j-1}(c_0)\ra=\mathcal{O}_L
= \la a_{j+\ell-1}(c_0)-w' a_{j-1}(c_0)\ra. \]

Therefore, we may assume for the rest of the proof that $\ell\equiv 0(\text{mod }n)$.
Write $\ell=nt$ for some $t\in \NN$, and as usual, write $f:=f_{d,c_0}$.
We proceed via a local argument.

For any place $v$ of $L$ that does not divide $d$, we have
\[ |a_{j+nt-1}(c_0)-w a_{j-1}(c_0)|_v=1 = |a_{j+nt-1}(c_0)-w' a_{j-1}(c_0)|_v \]
by Lemma~\ref{lem:div-d}.
For the rest of the proof, then, let $v$ be a place of $L$ that divides $d$,
and let $\rho:=|p|_v^{p/(d(p-1))}$ and $\kappa:=|p|_v^{1/(p-1)}$,
which are the maximum and minimum $v$-adic distances (respectively)
between a nontrivial $d$-th root of unity and $1$. 

For any $x\in\Cv$ with $|x|_v < \rho$, expanding $(1+x)^d$ shows that
\[ \big|(1+x)^d - 1\big|_v < \rho^d = |p|_v \kappa . \] 
Thus, for any $b,c\in\Cv^{\times}$ with $|b-c|_v < \rho |b|_v$, we have
\begin{equation}
\label{eq:fmapbound}
\big| f(b) - f(c)\big|_v < |p|_v \kappa |b|_v^d .
\end{equation}
We claim that for any $d$-th root of unity $\eta$, we have
\begin{equation}
\label{eq:claimdisk}
\big| a_{j+nt-1}(c_0) - \eta a_{j-1}(c_0) \big|_v \geq \rho \big|a_{j-1}(c_0)\big|_v .
\end{equation}
To prove the claim, suppose inequality~\eqref{eq:claimdisk} fails for some such $\eta$.
Then by inequality~\eqref{eq:fmapbound} with $b=\eta a_{j-1}(c_0)$ and $c=a_{j+nt-1}(c_0)$,
we have
\[ \big| a_{j+nt}(c_0) - a_j(c_0) \big|_v < |p|_v \kappa \big| a_{j-1}(c_0) \big|_v^d \leq |p|_v \kappa, \]
since $|a_{j-1}(c_0)|_v \leq 1$. Applying $m-j-1\geq 0$ more iterations of $f$,
and noting that $f$ does not expand distances on $\Dbar(0,1)$, we have
\begin{equation}
\label{eq:indisk}
\big| a_{m+nt-1}(c_0) - a_{m-1}(c_0) \big|_v < |p|_v \kappa .
\end{equation}

However,  by Theorem~\ref{thm:vefa3.1}, we have
\[ v\big(a_i(c_0) \big) \leq v(p), \quad \text{i.e.,} \quad \big|a_i(c_0)\big|_v \geq |p|_v
\quad \text{for all } i\geq 1, \]
since $d^{m-1}-1\geq 2^1 - 1 \geq 1$.
Thus, inequality~\eqref{eq:indisk} yields
$|a_{m+nt-1}(c_0) - a_{m-1}(c_0)|_v < \kappa |a_{m-1}(c_0)|_v$.
However, $a_{m+nt-1}(c_0) = \zeta a_{m-1}(c_0)$, where $\zeta\neq 1$ is a $d$-th root of unity. Therefore,
since $\kappa\leq |1-\zeta|_v$, we have
\[ \kappa |a_{m-1}(c_0)|_v \leq |\zeta a_{m-1}(c_0) - a_{m-1}(c_0)|_v < \kappa |a_{m-1}(c_0)|_v.\]
This contradiction proves the claim of inequality~\eqref{eq:claimdisk}.

We now use the claim to prove the lemma. Observe that
\begin{equation}
\label{eq:omegarho}
\big| w' a_{j-1}(c_0) - w a_{j-1}(c_0) \big|_v =
|w' - w|_v \big| a_{j-1}(c_0) \big|_v \leq \rho \big| a_{j-1}(c_0) \big|_v
\leq \big| a_{j+nt-1}(c_0) - w a_{j-1}(c_0) \big|_v ,
\end{equation}
where the first inequality is by definition of $\rho$, and the second is by the claim
applied to $\eta=w$. Therefore,
\begin{align*}
\big| a_{j+nt-1}(c_0) - w' a_{j-1}(c_0) \big|_v &\leq 
\max\big\{ \big| a_{j+nt-1}(c_0) - w a_{j-1}(c_0)\big|_v,
\big| w' a_{j-1}(c_0) - w a_{j-1}(c_0) \big|_v \big\}\\
&\leq \big| a_{j+nt-1}(c_0) - w a_{j-1}(c_0) \big|_v ,
\end{align*}
where the first inequality is the non-archimedean triangle inequality,
and the second is by inequality~\eqref{eq:omegarho}.

We have just shown that 
$| a_{j+nt-1}(c_0) - w' a_{j-1}(c_0)|_v \leq | a_{j+nt-1}(c_0) - w a_{j-1}(c_0) |_v$.
Applying the same argument with the roles of $w$ and $w'$ reversed,
we similarly have
\[ \big| a_{j+nt-1}(c_0) - w a_{j-1}(c_0) \big|_v \leq \big| a_{j+nt-1}(c_0) - w' a_{j-1}(c_0) \big|_v , \]
thus proving the lemma.
\end{proof}


\begin{proof}[Proof of Theorem~\ref{thm:jnotm} for $j<m$]
We will consider the cases $\ell\not\equiv 0\pmod{n}$ and $\ell\equiv 0\pmod{n}$ separately.

\textbf{Case 1}.
Suppose that $\ell\not\equiv 0\pmod{n}$. The result is immediate from part (a) of Proposition~\ref{prop:3.2}, because by the M\"{o}bius product
definition of $G_{d,j,\ell}^{\zeta}$, we have
\[ \big\la G_{d,j,\ell}^{\zeta}(c_0)\big\ra \, \big| \,
\big\la a_{j+\ell-1}(c_0)-\zeta a_{j-1}(c_0)\big\ra \]
as ideals in $\mathcal{O}_K$.

\medskip

\textbf{Case 2}.
Suppose that $\ell\equiv 0\pmod{n}$. Set $\ell=nt$ for some $t\in \mathbb{N}$, and $L:=\mathbb{Q}(c_0,\eta)$ for some primitive $d$-th root of unity $\eta$. By \cite[Lemma 27]{BEK19}, there is a polynomial $F\in \mathbb{Z}[c]$ such that
	\begin{equation}
	\label{eq:10}
	\prod_{\substack{w^d=1\\ w\neq 1}} G_{d,j,nt}^w(c_0) = G_{d,0,nt}(c_0)^{(d-1)N_{j,nt}}+pF(c_0).
	\end{equation}
First consider the case $t>1$.
By equation~\eqref{eq:Mobius_Gleason},
we know that $G_{d,0,nt}(c_0)$ divides $\frac{a_{nt}(c_0)}{a_n(c_0)}$ in $\calO_K$.
(See also Lemma~5.4 of \cite{Looper19}.)
By Corollary~\ref{cor:2.4}, it follows that $u_1:=G_{d,0,nt}(c_0)$ is a unit in $\calO_K$.




Substituting this value in \eqref{eq:10}, we obtain
\begin{equation}
\label{eq:Gleason_unit}
\prod_{\substack{w^d=1\\ w\neq 1}} G_{d,j,nt}^w(c_0) = u_1^{(d-1)N_{j,nt}}+pF(c_0)
\end{equation}
Since we have
\[ \big\la G_{d,j,nt}^{w}(c_0)\big\ra \, \big| \,
\big\la a_{j+nt-1}(c_0)-w a_{j-1}(c_0)\big\ra \]
as ideals in $\mathcal{O}_L$, if there were a prime ideal $\mathfrak{p}\subseteq \mathcal{O}_L$ such that
$ G_{d,j,nt}^{w}(c_0)\equiv 0\pmod{\mathfrak{p}}$, then Lemma~\ref{lem:div-d} yields
$p\equiv 0\pmod{\mathfrak{p}}$.
This fact together with \eqref{eq:Gleason_unit} implies $u_1\equiv 0 \pmod{\mathfrak{p}}$, a contradiction.
Hence, there is no such a prime ideal $\mathfrak{p}\subseteq \mathcal{O}_L$,
whence $\la G_{d,j,nt}^w(c_0)\ra=\mathcal{O}_L$ for each $d$-th root of unity $w$.
In particular, we have $\la G_{d,j,\ell}^{\zeta}(c_0)\ra=\mathcal{O}_K$,
completing the proof of part (a) of Theorem~\ref{thm:jnotm} for $j<m$.

It remains to consider the case that $t=1$, i.e. $\ell=n$.
By \cite[Lemma 2.2]{Gok20}, there is a unit $u_2$ in $\calO_K$ such that $G_{d,0,n}(c_0) = u_2a_n(c_0)$. Substituting this value in \eqref{eq:10}, we obtain
\[ \prod_{\substack{w^d=1\\ w\neq 1}} G_{d,j,n}^w(c_0) = u_3 a_n(c_0)^{(d-1)N_{j,n}}+pF(c_0) \]
for some unit $u_3$ in $\calO_K$.

Define $E,M$ as in equations~\eqref{eq:EMdef}, and observe that $EM > (d-1) N_{j,n}$.
Recall from Theorem~\ref{thm:vefa3.1} that $\la a_n(c_0)\ra^{EM} = \la p\ra$.
Hence, there exists a unit $u_4$ in $\calO_K$ such that
\begin{equation}
\label{eq:Gwprod}
\prod_{\substack{w^d=1\\ w\neq 1}}  G_{d,j,n}^w(c_0)
=  a_n(c_0)^{(d-1)N_{j,n}}\big(u_3+u_4 a_n(c_0)^{EM-(d-1)N_{j,n}} F(c_0)\big).
\end{equation}

For each $w$ in the above product, $\la G_{d,j,n}^w(c_0)\ra$ divides
$\la a_{j+n-1}(c_0)-wa_{j-1}(c_0)\ra$ (as ideals in $\mathcal{O}_L$),
by the M\"{o}bius product definition of $G_{d,j,n}^w$.
By part~(b) of Proposition~\ref{prop:3.2} and by Lemma~\ref{lem:primitive-w}, then,
any prime ideal of $\mathcal{O}_L$ dividing $\la G_{d,j,n}^w(c_0)\ra$ must divide $\la a_n(c_0)\ra$.
By equation~\eqref{eq:Gwprod},
any prime ideal $\frp\subseteq \mathcal{O}_L$ dividing $\la u_3+u_4 a_n(c_0)^{EM-(d-1)N_{j,n}} F(c_0) \ra$
must divide some $\la G_{d,j,n}^w(c_0)\ra$ and hence also divides $\la a_n(c_0)\ra$.
Then
\[ u_3 \equiv u_3+u_4 a_n(c_0)^{EM-(d-1)N_{j,n}} F(c_0) \equiv 0 \pmod{\frp} ,\]
contradicting the fact that $u_3$ is a unit, and hence showing that no such $\frp$ exists.

Thus, $u_3+u_4 a_n(c_0)^{EM-N_{j,n}} F(c_0)$ must be a unit in $\calO_L$,
and hence also in $\mathcal{O}_K$. Therefore,
\[ \prod_{\substack{w^d=1\\ w\neq 1}}\la G_{d,j,n}^w(c_0)\ra = \la a_n(c_0)\ra^{(d-1)N_{j,n}} \]
as ideals in $\mathcal{O}_K$.
Finally, by Lemma~\ref{lem:primitive-w} and the definition of $G_{d,j,n}^w$,
for any $w,w'$ with $w^d=(w')^{d} = 1$ and $w, w'\neq 1$,
we have $\la G_{d,j,n}^w(c_0)\ra = \la G_{d,j,n}^{w'}(c_0)\ra$ (as ideals in $\mathcal{O}_L $).
The desired equality $\la G_{d,j,n}^{\zeta}(c_0)\ra = \la a_n(c_0)\ra^{N_{j,n}}$ follows immediately.
\end{proof}

\section{$G_{d,j,\ell}^{\zeta}(c_0)$ when $j>m$}
\label{sec:j>m}
In this section, we answer Question~\ref{ques:unit} for $j>m$ by proving Theorem~\ref{thm:jnotm} in that case.
\begin{lemma}
\label{lem:unit_j>m}
Let $m\geq 2$ and $n\geq 1$, and let $d=p^e$, where $p$ is a prime and $e\geq 1$. Let $c_0$ be a root of $G_{d,m,n}^{\zeta}$, where $\zeta\neq 1$ is a $d$-th root of unity. Set $K=\mathbb{Q}(c_0)$. Suppose that $j>m$ and $\ell\geq 1$.
\begin{enumerate}[label=\textup{(\alph*)}]
	\item If $\ell\not\equiv 0 \pmod{n}$, then $\la a_{j+\ell-1}(c_0)-\zeta a_{j-1}(c_0)\ra=\mathcal{O}_K$.
	\item If $\ell\equiv 0 \pmod{n}$, then $\la a_{j+\ell-1}(c_0)-\zeta a_{j-1}(c_0)\ra=\la (1-\zeta)a_{j-1}(c_0)\ra$ as ideals in $\mathcal{O}_K$.
\end{enumerate}
\end{lemma}
\begin{proof}
\textbf{Case (a)}.
Suppose for the sake of contradiction that there exists a prime ideal $\mathfrak{p}\subseteq \mathcal{O}_K$ that satisfies $a_{j+\ell-1}(c_0)-\zeta a_{j-1}(c_0)\equiv 0 \pmod{\mathfrak{p}}$, i.e.
\begin{equation}
\label{eq:j>m_zeta}
a_{j+\ell-1}(c_0)\equiv \zeta a_{j-1}(c_0) \pmod{\mathfrak{p}}.
\end{equation} 
Applying $n$ iterations of $f:=f_{d,c_0}$ to both sides of (\ref{eq:j>m_zeta}), we obtain
\[ a_{j+\ell+n-1}(c_0)\equiv a_{j+n-1}(c_0)\pmod{\mathfrak{p}}. \]
Since $f$ has exact type $(m,n)$ and $j-1\geq m$, it follows that
\begin{equation}
\label{eq:j>m_zeta2}
a_{j+\ell-1}(c_0)\equiv a_{j-1}(c_0)\pmod{\mathfrak{p}}.
\end{equation}
Combining (\ref{eq:j>m_zeta}) and (\ref{eq:j>m_zeta2}) yields
\begin{equation}
\label{eq:ajminus}
(1-\zeta)a_{j-1}(c_0)\equiv 0\pmod{\mathfrak{p}}.
\end{equation}
By Theorem~\ref{thm:1.1}, we have
$\la 1-\zeta\ra |\la p\ra$ and $\la a_{j-1}(c_0)\ra |\la p\ra$,
and hence equation~\eqref{eq:ajminus} implies that $p\equiv 0\pmod{\mathfrak{p}}$.
Theorem~\ref{thm:1.1} and equation~\eqref{eq:ajminus} together also force
\begin{equation}
\label{eq:p|a_n}
a_{nt}(c_0)\equiv 0\pmod{\mathfrak{p}}
\quad \text{for any positive integer } t.
\end{equation}
Let $t$ be a positive integer with $nt\geq j$.
Applying $nt-j+1$ iterations of $f$ to both sides of \eqref{eq:j>m_zeta2} yields
\begin{equation}
\label{eq:antlant}
a_{nt+\ell}(c_0)\equiv a_{nt}(c_0)\pmod{\mathfrak{p}}.
\end{equation}
However, because $\ell\not\equiv 0\pmod{n}$,
Theorem~\ref{thm:1.1} implies that $a_{nt+\ell}(c_0)$ is a unit in $\mathcal{O}_K$;
thus, equations~\eqref{eq:p|a_n} and~\eqref{eq:antlant} contradict one another.
Hence, there is no such prime ideal $\mathfrak{p}\subseteq \mathcal{O}_K$.
That is, $\la a_{j+\ell-1}(c_0)-\zeta a_{j-1}(c_0)\ra=\mathcal{O}_K$, as desired.

\medskip

\textbf{Case (b)}. If $\ell\equiv 0\pmod{n}$, then because $f$ has exact type $(m,n)$ and $j-1\geq m$,
we obtain $a_{j+\ell-1}(c_0)=a_{j-1}(c_0)$, which immediately implies the result.
\end{proof}

\begin{proof}[Proof of Theorem~\ref{thm:jnotm} for $j>m$.]
We again consider the cases $\ell\not\equiv 0\pmod{n}$ and $\ell\equiv 0\pmod{n}$ separately.

\textbf{Case 1}.
Suppose that $\ell\not\equiv 0\pmod{n}$. The result is immediate from part (a) of Lemma~\ref{lem:unit_j>m}, because by the M\"{o}bius product
definition of $G_{d,j,\ell}^{\zeta}$, we have
\[ \big\la G_{d,j,\ell}^{\zeta}(c_0)\big\ra \, \big| \,
\big\la a_{j+\ell-1}(c_0)-\zeta a_{j-1}(c_0)\big\ra \]
as ideals in $\mathcal{O}_K$.
\medskip

\textbf{Case 2}.
Suppose that $\ell\equiv 0\pmod{n}$. Write $\ell=nt$ for some $t\in\mathbb{N}$.
We first consider the case $nt\nmid j-1$. By definition, we have
\[ G_{d,j,nt}^{\zeta}(c_0) = \prod_{k|nt} (a_{j+k-1}(c_0)-\zeta a_{j-1}(c_0) )^{\mu(nt/k)}. \]
By Lemma~\ref{lem:unit_j>m}, $a_{j+k-1}(c_0)-\zeta a_{j-1}(c_0)$ is a unit in $\mathcal{O}_K$
for each $k\not\equiv 0\pmod n$. Thus,
\begin{align*}
\la G_{d,j,nt}^{\zeta}(c_0)\ra
&= \prod_{\substack{k|nt \\ n|k}} \la a_{j+k-1}(c_0)-\zeta a_{j-1}(c_0)\ra^{\mu(nt/k)}
  =\prod_{k_1|t} \la a_{j+nk_1-1}(c_0)-\zeta a_{j-1}(c_0)\ra^{\mu(t/k_1)}\\
&=\big\la(1-\zeta)a_{j-1}(c_0)\big\ra^{\sum_{k_1|t}\mu(t/k_1)}
 =\begin{cases}
 \mathcal{O}_K & \text{if } t>1\\
  \la (1-\zeta)a_{j-1}(c_0)\ra & \text{if } t=1
 \end{cases}\\
 &=\begin{cases}
 \mathcal{O}_K & \text{if } \ell>n\\
 \la 1-\zeta\ra & \text{if } \ell=n
 \end{cases}
\end{align*}
as desired. In particular, the third equality is by Lemma~\ref{lem:unit_j>m},
the fourth is by the M\"{o}bius identity~\eqref{eq:Mobius},
and the fifth is by Theorem~\ref{thm:1.1} together with the fact that $n\nmid j-1$.

It remains to consider the case $nt  \,\mid\, j-1$.
Using the first part of Case 2 and the definition of $G_{d,j,nt}^{\zeta}$, we have
\begin{equation}
\label{eq:nt|j-1}
\la G_{d,j,nt}^{\zeta}(c_0)\ra = \begin{cases}
\prod_{k|nt} \la a_k(c_0)\ra^{-\mu(nt/k)} & \text{if } t>1\\
\la (1-\zeta)a_{j-1}(c_0)\ra \prod_{k|nt} \la a_k(c_0)\ra^{-\mu(nt/k)} & \text{if } t=1
\end{cases}
\end{equation}
as ideals in $\mathcal{O}_K$. By Theorem~\ref{thm:1.1}, (\ref{eq:nt|j-1}) immediately yields
\begin{align*}
\big\la G_{d,j,nt}^{\zeta}(c_0)\big\ra &=  \begin{cases}
\la a_n(c_0)\ra^{-\sum_{k_1|t}\mu(t/k_1)} & \text{if } t>1\\
\la (1-\zeta)a_{j-1}(c_0)\ra \la a_n(c_0)\ra^{-\sum_{k_1|t}\mu(t/k_1)} & \text{if } t=1 \end{cases}\\
&= \begin{cases}
\mathcal{O}_K & \text{if } \ell>n\\
\la 1-\zeta\ra & \text{if } \ell=n \end{cases}
\end{align*}
as desired.
Note that in the last equality, we used Theorem~\ref{thm:1.1}, equation~\eqref{eq:Mobius},
and the fact that $n  \,\mid\, j-1$.
\end{proof}

\bigskip

\textbf{Acknowledgments.}
The first author gratefully acknowledges the support of NSF grant DMS-2101925.

\end{document}